\documentclass[11pt]{article}
\usepackage[margin=1in]{geometry}
\usepackage{amsmath,amssymb,amsthm}
\usepackage{booktabs}
\usepackage{array}
\usepackage{microtype}
\usepackage[hidelinks]{hyperref}
\usepackage{enumitem}
\usepackage[T1]{fontenc}
\usepackage{lmodern}

\newtheorem{theorem}{Theorem}
\newtheorem{proposition}[theorem]{Proposition}
\newtheorem{lemma}[theorem]{Lemma}
\newtheorem{corollary}[theorem]{Corollary}
\theoremstyle{definition}
\newtheorem{definition}[theorem]{Definition}
\theoremstyle{remark}
\newtheorem{remark}[theorem]{Remark}

\newcommand{\CT}{\operatorname{CT}}

\hypersetup{
  pdftitle={The Mirror Theorem: Central Incidence Saturation in Keith Recurrences},
  pdfauthor={Federico Ignacio Zamponi},
  pdfsubject={Keith numbers, incidence saturation, and exact cross-base classification},
  pdfkeywords={Keith numbers, repfigit numbers, incidence saturation, central deficit, triple alignment}
}

\title{\textbf{The Mirror Theorem: Central Incidence Saturation in Keith Recurrences}}
\author{Federico Ignacio Zamponi\\\small Independent Researcher, Buenos Aires, Argentina}
\date{September 2026}

\begin{document}
\maketitle

\begin{abstract}
For an order-$d$ sliding-window recurrence, let $c_i$ count the recurrence equations through an even index $m$ in which the term $K_i$ occurs, and let $C_d(m)$ be the sum of the two central incidence counts. We prove the exact structural law
\[
\Delta_d(m):=m-C_d(m)=\operatorname{dist}\!\bigl(m,[2d,2d+2]\bigr).
\]
Thus the zero-defect condition $C_d(m)=m$ localizes every even candidate index to $m\in\{2d,2d+2\}$. We call classical Keith numbers whose first return satisfies this condition \emph{Keith mirrors}. For bases $B\ge3$ and orders $d\ge3$, the early boundary is impossible, so every higher-order mirror returns at $R=2d+2$. At that boundary an explicit digit identity and coefficient-capacity bounds reduce the problem to a finite region. An exact integer audit, independently checked against the original recurrence and incidence definition, yields precisely $35$ higher-order mirrors; order two consists of two explicit infinite families.

For the boundary-digit observable $E_B=x_1+x_d$, we further prove that the triple condition $R=C=E_B$ forces $d\le3$ and has exactly eight cross-base realizations. Its terminal profile is $1,2,3,4\mid4,3,2,1$. Finally, the incidence symmetry gives the fixed-cut identity $X=A\mid\operatorname{rev}(A)$, recorded only as a structural interpretation: exact reconstruction requires both the cut and the global mirror law.
\end{abstract}

\noindent\textbf{Keywords.} Keith numbers; repfigit numbers; incidence saturation; central deficit; incidence profiles; triple alignment; Diophantine equations; computational classification.

\medskip
\noindent\textbf{2020 Mathematics Subject Classification.} Primary 11B39; Secondary 11A63, 11B50.
\section{Introduction}

A Keith number (also called a repfigit number) is obtained by using the digits of an integer as the initial terms of a Fibonacci-like recurrence and asking that the integer itself later reappear in the resulting sequence. The notion was introduced by Mike Keith in 1987 \cite{Keith1987}. Subsequent computational work treated the fixed-digit problem through bounded linear Diophantine equations; Keith's later account makes this reduction explicit \cite{KeithWeb}, and Lichtblau developed lattice and integer-programming methods for the resulting knapsack problems \cite{Lichtblau2006}. Nondecimal repfigits were already considered by Esche \cite{Esche1994}. Klazar and Luca proved, among other results, that decimal Keith numbers have asymptotic density zero and that an $n$-digit decimal Keith number with $n\ge3$ can return only at an index $m$ satisfying $2n<m<7n$ \cite{KlazarLuca2007}; they also discuss extensions of their arguments to fixed nondecimal bases.

The present paper studies a different question. Rather than searching over all possible return indices, we attach an \emph{incidence profile} to the recurrence equations used before the first return. The profile depends only on the sliding-window architecture of the recurrence and is symmetric about its center. We ask for the extremal case in which the sum of the two central incidence counts equals the return index itself. We call the resulting incidence-saturated objects \emph{Keith mirrors}.

The structural restriction is unexpectedly strong. For every even candidate index $m$, define the central deficit
\[
\Delta_d(m):=m-C_d(m).
\]
The first main result is the exact law
\[
\Delta_d(m)=\operatorname{dist}\!\bigl(m,[2d,2d+2]\bigr).
\]
Thus incidence saturation is not an isolated equality: it is precisely the zero-defect stratum, and it forces the return to one of only two positions, $2d$ or $2d+2$. For orders $d\ge3$ the first branch is impossible in every base $B\ge3$, so the return is rigidly fixed at $2d+2$. At that boundary the Keith condition collapses to a single explicit Diophantine identity in the digits. A comparison between the positive excess of the leading digit and the total compensating capacity of all negative coefficients then reduces the apparently infinite problem to a finite region in the $(B,d)$-plane. The remaining cases can be solved exactly.

The principal conclusion is that there are exactly $35$ Keith mirrors of order at least $3$. The order-$2$ case behaves differently: it produces explicit families in infinitely many bases and is therefore treated separately.

A second question returns to the observation that motivated the construction. For
\[
E_B(N):=x_1+x_d,
\]
we ask when a mirror also satisfies $E_B=R$. This \emph{triple alignment} is much more rigid than incidence saturation alone: it has exactly eight realizations across all bases $B\ge3$, and no realization of order $d\ge4$. Thus the symmetric profile
\[
1,2,3,4\mid4,3,2,1
\]
is terminal for the triple-alignment condition. The familiar decimal example $197$ is one member of this terminal cross-base layer.

The central cut also suggests a limited structural interpretation. If an even incidence profile is written as $X=A\mid B$, then the incidence symmetry forces $B=\operatorname{rev}(A)$. The paper uses this only as a concrete instance of conditional part--whole reconstruction: the observed half determines the complement once both the cut and the global mirror constraint are known. No claim is made that unrestricted local data determine a global object.

\section{Relation to prior work and scope}

The linear Diophantine viewpoint on Keith numbers is not new. Keith's computational work already treats fixed-length searches as bounded Diophantine problems \cite{KeithWeb}; Lichtblau expresses candidate returns using powers of a Fibonacci matrix and solves the resulting bounded integer equations by lattice reduction and integer programming \cite{Lichtblau2006}. Likewise, the rarity of Keith numbers and the existence of finite subclasses are already part of the literature: Klazar and Luca prove asymptotic density zero and effective finiteness for repdigit Keith numbers in fixed bases \cite{KlazarLuca2007}, and Bravo, Guzm\'an, and Luca later prove that no decimal repdigit Keith numbers exist \cite{BravoGuzmanLuca2013}. Nondecimal variants also precede the present work \cite{Esche1994}.

The contribution here is narrower and structural. We study the incidence-saturated slice $C=R$, prove that it rigidly fixes the return boundary, derive the closed identity specific to that boundary, and use its coefficient geometry to prove finite support and complete classification. In particular, the paper does \emph{not} claim novelty for converting Keith conditions into linear Diophantine equations, for bounded-digit enumeration, or for min/max pruning as a general computational idea. The distinctive use of capacity here is as a closed analytic inequality that eliminates entire infinite regions of the $(B,d)$-plane before any residual enumeration is performed.

Klazar and Luca's decimal return bound $2d<m$ for $d\ge3$ already excludes the branch $m=2d$ in base $10$ \cite{KlazarLuca2007}. In that setting, $2d+2$ is therefore the earliest \emph{even} return index allowed by the known bound. The argument below proves the exclusion of $2d$ uniformly for every base $B\ge3$, which is essential for the cross-base classification.

\paragraph{Terminology and novelty scope.}
We introduce the term \emph{Keith mirror} for the incidence-based subclass defined below. The name refers to the symmetry of the incidence profile, not to reversal of the digit string. It should therefore not be confused with a \emph{reverse Keith number} (or revrepfigit), for which the reversed digit string appears in the generated recurrence \cite{OEISReverse}. In the standard Keith-number sources and indexed literature reviewed for this draft, we found no prior use of the term \emph{Keith mirror} and no equivalent subclass defined by the condition $C=R$. Because a negative literature search cannot establish historical priority, the appropriate claim is narrower: to the best of our knowledge, this incidence-saturated subclass has not previously been isolated and named. The paper also does not claim that the individual digit strings in Table~\ref{tab:mirrors} are newly discovered Keith numbers; the contribution is their structural characterization and complete classification under the mirror condition.

We work throughout with $B\ge3$. The binary case is outside the scope of the theorem; Klazar and Luca note a separate unpublished construction of binary Keith numbers due to Kenneth Fan \cite{KlazarLuca2007}. The classification below is analytic in its infinite part: the exclusion of all sufficiently large bases and orders is proved by inequalities. A finite exact audit is used only after the support theorem has reduced the problem to a small bounded set of base-order pairs.

The boundary-digit refinement and the part--whole notation are likewise kept narrow. Within the sources reviewed for this draft we found no prior treatment of the condition $R=C=E_B$, but this negative search does not establish historical priority. The reconstruction statement $A\mid B$ with $B=\operatorname{rev}(A)$ is an immediate consequence of the incidence symmetry once the cut is fixed; its role here is interpretive, not a claim of a new general theory of inference.

\section{Definitions}

Let $B\ge3$ and let
\[
N=(x_1x_2\cdots x_d)_B
\]
be a $d$-digit base-$B$ integer, with
\[
x_1\in\{1,\ldots,B-1\},\qquad x_j\in\{0,\ldots,B-1\}\quad (j\ge2).
\]
Its order-$d$ Keith sequence $(K_n)_{n\ge1}$ is defined by
\[
K_1=x_1,\ldots,K_d=x_d,
\]
and, for $n>d$,
\[
K_n=K_{n-1}+K_{n-2}+\cdots+K_{n-d}.
\]
If $K_m=N$ for some $m>d$, then $N$ is a base-$B$ Keith number. We write
\[
R(N)=\min\{m>d:K_m=N\}
\]
for its first return index.

Fix any index $m>d$. The recurrence equations through $m$ are
\[
K_j=K_{j-1}+\cdots+K_{j-d},\qquad d+1\le j\le m.
\]
For $1\le i\le m$, let $c_i$ be the number of these equations in which $K_i$ occurs, counting an occurrence either as a summand or as the result. The vector
\[
(c_1,c_2,\ldots,c_m)
\]
is the \emph{incidence profile through $m$}. It is defined whether or not $m$ is a return index.

When $m$ is even, define the \emph{central incidence sum} and \emph{central deficit}
\[
C_d(m):=c_{m/2}+c_{m/2+1},
\qquad
\Delta_d(m):=m-C_d(m).
\]
If $N$ is a Keith number with even first return $R(N)$, we write
\[
C(N):=C_d(R(N)).
\]

\section{Incidence geometry and return rigidity}

\begin{lemma}[Exact incidence profile]\label{lem:profile}
For $1\le i\le m$,
\[
c_i=\min\{i,d+1,m-d,m-i+1\}.
\]
In particular,
\[
c_i=c_{m-i+1}.
\]
\end{lemma}

\begin{proof}
The term $K_i$ occurs in the equation indexed by $j$ exactly when
\[
j-d\le i\le j,
\]
with $d+1\le j\le m$. Hence the admissible equation indices are the integers in
\[
[d+1,m]\cap[i,i+d].
\]
Counting this interval gives the displayed minimum. Replacing $i$ by $m-i+1$ leaves the minimum unchanged, proving the mirror symmetry.
\end{proof}

\begin{definition}[Keith mirror]\label{def:mirror}
A base-$B$ Keith number is called a \emph{Keith mirror} if its first return index is even and
\[
C_d(R(N))=R(N).
\]
A Keith mirror of order $d\ge3$ will be called a \emph{higher-order Keith mirror} when the distinction from the exceptional order-$2$ case is useful.
\end{definition}

\begin{remark}[Why ``mirror'']
The terminology is structural: Lemma~\ref{lem:profile} shows that the full incidence profile is palindromic. The defining equation $C=R$ then selects the case in which the two central incidences saturate the largest value permitted by the return index. No reversal of the digits of $N$ is involved.
\end{remark}

\paragraph{Motivating example: $197$.}
In base $10$, the Keith sequence generated by $197$ begins
\[
1,9,7,17,33,57,107,197.
\]
The recurrence equations through the first return give the incidence profile
\[
1,2,3,4\mid4,3,2,1.
\]
Hence $R(197)=8$ and $C(197)=4+4=8$, so $197$ is a Keith mirror. The theorem below shows that this example belongs to a completely classifiable cross-base family.

\begin{theorem}[Exact central deficit law]\label{thm:deficit}
For every even index $m=2r>d$,
\[
\Delta_d(m)=
\begin{cases}
2d-m, & m<2d,\\
0, & 2d\le m\le2d+2,\\
m-2d-2, & m>2d+2.
\end{cases}
\]
Equivalently,
\[
\boxed{\;\Delta_d(m)=\operatorname{dist}\!\bigl(m,[2d,2d+2]\bigr).\;}
\]
\end{theorem}

\begin{proof}
By Lemma~\ref{lem:profile} and central symmetry,
\[
C_d(2r)=2\min\{r,d+1,2r-d\}.
\]
If $r<d$, the minimum is $2r-d$, so
\[
\Delta_d(2r)=2r-2(2r-d)=2d-2r=2d-m.
\]
If $d\le r\le d+1$, the minimum is $r$, giving $\Delta_d(2r)=0$. If $r>d+1$, the minimum is $d+1$, so
\[
\Delta_d(2r)=2r-2d-2=m-2d-2.
\]
These are exactly the distances from $m$ to the interval $[2d,2d+2]$.
\end{proof}

\begin{corollary}[Two-boundary rigidity]\label{cor:saturation}
For even $m>d$,
\[
C_d(m)=m
\]
if and only if
\[
2d\le m\le2d+2.
\]
Hence an even return satisfying incidence saturation must have
\[
\boxed{\;m\in\{2d,2d+2\}.\;}
\]
More generally, if $\Delta_d(m)=2h$, then
\[
m=2d-2h\qquad\text{or}\qquad m=2d+2+2h.
\]
\end{corollary}

\begin{proof}
Immediate from Theorem~\ref{thm:deficit}; because $m$ is even, the zero-defect interval $[2d,2d+2]$ contains only the two integer candidates $2d$ and $2d+2$. The fixed-deficit statement follows by solving the two nonzero branches of the same formula.
\end{proof}

\section{Two boundary formulas}

It is convenient to index the digits from the right. Write
\[
N_B=\sum_{k=0}^{d-1}a_kB^k,
\]
where $a_{d-1}\ne0$ and the initial Keith seed, from left to right, is
\[
a_{d-1},a_{d-2},\ldots,a_0.
\]
Let
\[
A=\sum_{k=0}^{d-1}a_k.
\]

\begin{lemma}[Boundary coefficients]\label{lem:boundary}
For every order $d\ge2$,
\[
K_{2d}=2^{d-1}A-\sum_{k=1}^{d-1}2^{k-1}a_k,
\]
and
\[
K_{2d+2}=(2^{d+1}-1)A-\sum_{k=0}^{d-1}2^{k+1}a_k.
\]
\end{lemma}

\begin{proof}
For $1\le j\le d+1$, direct induction from the recurrence gives
\[
K_{d+j}=2^{j-1}A-\sum_{i=1}^{j-1}2^{j-1-i}x_i,
\]
where $x_i=a_{d-i}$. Taking $j=d$ gives the first formula. Taking $j=d+1$ gives
\[
K_{2d+1}=2^dA-\sum_{k=0}^{d-1}2^ka_k.
\]
The sliding-window identity
\[
K_{n+1}=2K_n-K_{n-d}
\]
then yields
\[
K_{2d+2}=2K_{2d+1}-K_{d+1}
=(2^{d+1}-1)A-\sum_{k=0}^{d-1}2^{k+1}a_k,
\]
since $K_{d+1}=A$.
\end{proof}

\section{The early branch is impossible for \texorpdfstring{$d\ge3$}{d >= 3}}

Define
\[
f_0=1-2^{d-1},
\]
and, for $1\le k\le d-1$,
\[
f_k=B^k-2^{d-1}+2^{k-1}.
\]
Then Lemma~\ref{lem:boundary} gives
\[
N_B-K_{2d}=\sum_{k=0}^{d-1}a_kf_k.
\]
The leading coefficient is
\[
f_{d-1}=B^{d-1}-2^{d-2}>0.
\]

\begin{lemma}[Early-return capacity]\label{lem:earlycapacity}
For every $B\ge3$ and $d\ge3$,
\[
f_{d-1}>(B-1)\sum_{\substack{0\le k\le d-2\\f_k<0}}(-f_k).
\]
Consequently,
\[
K_{2d}<N_B
\]
for every $d$-digit seed.
\end{lemma}

\begin{proof}
For $d=3$ one has
\[
(f_0,f_1,f_2)=(-3,B-3,B^2-2),
\]
so the maximum negative contribution is $3(B-1)$, while
\[
B^2-2>3(B-1)\qquad(B\ge3).
\]

For $d=4$,
\[
(f_0,f_1,f_2,f_3)=(-7,B-7,B^2-6,B^3-4).
\]
For $B\ge7$ the negative capacity is $7(B-1)$, which is smaller than $B^3-4$. For $B=3,4,5,6$, the pairs
\[
(f_3,D)=(23,22),(60,30),(121,36),(212,40)
\]
are obtained directly, where $D$ denotes maximum negative capacity.

For general $d$, ignoring the positive terms can only enlarge the negative capacity, and
\[
\sum_{\substack{0\le k\le d-2\\f_k<0}}(-f_k)
\le 2^{d-2}(2d-3).
\]
If $B\ge4$ and $d\ge5$, it is enough to observe that
\[
B^{d-1}-2^{d-2}>B\,2^{d-2}(2d-3)>(B-1)2^{d-2}(2d-3).
\]
The first inequality follows from
\[
2^{d-2}-\frac14>2d-3,
\]
which holds at $d=5$ and then inductively.

It remains to take $B=3$. For $d=5,6,7$, direct evaluation gives
\[
(f_{d-1},D)=(73,64),(227,162),(697,418).
\]
For $d\ge8$,
\[
D\le2^{d-1}(2d-3).
\]
Set
\[
F_d=3^{d-1}-2^{d-2}-2^{d-1}(2d-3).
\]
Then $F_8=459>0$, and for $d\ge8$,
\[
F_{d+1}-2F_d=3^{d-1}-2^{d+1}>0.
\]
Hence $F_d>0$ for every $d\ge8$, which is exactly
\[
3^{d-1}-2^{d-2}>2^{d-1}(2d-3).
\]
This completes the proof.
\end{proof}

\begin{corollary}[Return rigidity]\label{cor:rigidity}
Every higher-order Keith mirror satisfies
\[
R(N)=2d+2.
\]
\end{corollary}

\begin{proof}
Corollary~\ref{cor:saturation} leaves only $2d$ and $2d+2$. Lemma~\ref{lem:earlycapacity} excludes $2d$.
\end{proof}

\begin{lemma}[Automatic first return]\label{lem:firstreturn}
Let $B\ge3$ and $d\ge3$. If a $d$-digit seed satisfies
\[
K_{2d+2}=N_B,
\]
then $2d+2$ is automatically the first post-seed occurrence of $N_B$.
\end{lemma}

\begin{proof}
By Lemma~\ref{lem:earlycapacity}, $K_{2d}<N_B$. We first justify monotonicity on the required interval. Since $K_{d+1}=A$, suppose inductively that $K_n\ge A$ for some $d+1\le n\le2d$. Then $K_{n-d}$ is one of the initial digits and therefore $K_{n-d}\le A$, so the sliding-window identity gives
\[
K_{n+1}-K_n=K_n-K_{n-d}\ge0.
\]
Thus
\[
K_{d+1}\le K_{d+2}\le\cdots\le K_{2d+1}.
\]
In particular, every post-seed term through $K_{2d}$ is at most $K_{2d}<N_B$.

Finally,
\[
K_{2d+2}=K_{2d+1}+K_{2d}+\cdots+K_{d+2}.
\]
All post-seed terms in the displayed sum are positive, so $K_{2d+1}<K_{2d+2}=N_B$. Hence no earlier post-seed term equals $N_B$.
\end{proof}

\section{The mirror identity and capacity}

Define
\[
Y=\sum_{k=0}^{d-1}a_k2^k.
\]

\begin{theorem}[Mirror identity]\label{thm:identity}
For $B\ge3$ and $d\ge3$, a $d$-digit seed is a Keith mirror if and only if
\[
\boxed{\;N_B+2Y=(2^{d+1}-1)A\;}.
\]
Equivalently,
\[
\sum_{k=0}^{d-1}a_ke_k=0,
\]
where
\[
\boxed{\;e_k=B^k+2^{k+1}-(2^{d+1}-1).\;}
\]
\end{theorem}

\begin{proof}
By Corollary~\ref{cor:rigidity}, a mirror must return at $2d+2$. Lemma~\ref{lem:boundary} gives
\[
K_{2d+2}=(2^{d+1}-1)A-2Y.
\]
Thus $K_{2d+2}=N_B$ is precisely the displayed identity. Conversely, the identity gives $K_{2d+2}=N_B$, and Lemma~\ref{lem:firstreturn} makes this the first return. Corollary~\ref{cor:saturation} then gives $C=R$.
\end{proof}

Set
\[
T_d=2^{d+1}-1.
\]
For a solution of Theorem~\ref{thm:identity}, the leading coefficient is
\[
e_{d-1}=B^{d-1}-2^d+1>0.
\]

\begin{lemma}[Capacity criterion]\label{lem:capacity}
A necessary condition for a higher-order Keith mirror is
\[
\boxed{\;e_{d-1}\le(B-1)\sum_{e_k<0}(-e_k).\;}
\]
\end{lemma}

\begin{proof}
Because $a_{d-1}\ge1$, the positive side of
\[
\sum a_ke_k=0
\]
contains at least $e_{d-1}$. Every negative coefficient can contribute at most $(B-1)(-e_k)$ in absolute value. Therefore the total negative capacity must be at least the leading positive excess.
\end{proof}

\section{Finite support}

\begin{proposition}[Exclusion of bases $B\ge4$ at high order]\label{prop:B4}
If $B\ge4$ and $d\ge7$, then no higher-order Keith mirror exists.
\end{proposition}

\begin{proof}
For $B=4,d=7$, the coefficients are
\[
(-252,-247,-231,-175,33,833,3969),
\]
so the maximum negative capacity is
\[
3(252+247+231+175)=2715<3969=e_6.
\]

For $B=4,d=8$, the leading excess is
\[
L_{4,8}=4^7-2^8+1=16129,
\]
whereas the crude bound
\[
U_{B,d}=(B-1)(d-1)(2^{d+1}-1)
\]
gives
\[
U_{4,8}=10731.
\]
For fixed $B\ge4$,
\[
L_{B,d+1}>3L_{B,d},
\]
while, for $d\ge8$,
\[
\frac{U_{B,d+1}}{U_{B,d}}
=\frac{d}{d-1}\frac{2^{d+2}-1}{2^{d+1}-1}<3.
\]
Thus the inequality persists for all $d\ge8$.

For $B\ge5,d=7$ the same crude bound already gives
\[
B^6-2^7+1>(B-1)6(2^8-1),
\]
with the smallest case $B=5$ equal to $15498>6120$. The same growth comparison propagates the exclusion to all larger $d$.
\end{proof}

\begin{proposition}[Ternary extinction]\label{prop:ternary}
For $B=3$, no higher-order Keith mirror exists for $d\ge11$.
\end{proposition}

\begin{proof}
For $d=11$, the coefficient vector is
\[
(-4092,-4088,-4078,-4052,-3982,-3788,-3238,-1652,2978,16612,57002).
\]
The maximum negative capacity is
\[
2(4092+4088+4078+4052+3982+3788+3238+1652)=57940.
\]
Since $2\cdot57002>57940$, the leading digit must be $1$. Any positive digit at coefficient $16612$ gives total positive excess at least $73614$, and any positive digit at coefficient $2978$ gives at least $59980$; hence those two digits must be zero. The negative side must therefore equal $57002$.

At full negative capacity the value is $57940$, so one would have to remove exactly
\[
57940-57002=938.
\]
But decreasing any negative digit from its maximal value $2$ changes the total by at least $1652$. Hence $938$ is impossible.

For $d=12$, direct capacity gives
\[
e_{11}=3^{11}-2^{12}+1=173052,
\]
while the exact negative capacity is $125712$.

For $d\ge13$, let
\[
L_d=3^{d-1}-2^d+1,
\qquad
U_d=2(d-1)(2^{d+1}-1).
\]
At $d=13$,
\[
L_{13}=523250>393192=U_{13}.
\]
Moreover,
\[
L_{d+1}>\frac52L_d,
\]
because
\[
L_{d+1}-\frac52L_d
=\frac12\,3^{d-1}+\frac12\,2^d-\frac32>0,
\]
whereas
\[
\frac{U_{d+1}}{U_d}
=\frac{d}{d-1}\frac{2^{d+2}-1}{2^{d+1}-1}<\frac52
\]
for $d\ge13$. Thus the capacity inequality fails for every $d\ge13$.
\end{proof}

\begin{proposition}[Low-order support]\label{prop:low-support}
Any higher-order Keith mirror lies in the following finite support:
\[
\begin{array}{c|c}
 d & \text{possible bases before exact digit solving}\\
\hline
3 & 3\le B\le11\\
4 & 3\le B\le6\\
5,6 & 3\le B\le4\\
7,8,9,10 & B=3\\
\ge11 & \varnothing.
\end{array}
\]
\end{proposition}

\begin{proof}
The last two rows are Propositions~\ref{prop:B4} and \ref{prop:ternary}. For $d=3$ the coefficients are
\[
(-12,B-11,B^2-7).
\]
For $B\ge12$, the only negative coefficient is $-12$, and
\[
B^2-7>12(B-1),
\]
so capacity fails.

For $d=4$ the coefficients are
\[
(-28,B-27,B^2-23,B^3-15).
\]
For $B\ge7$, the coefficient $B^2-23$ is positive and the negative capacity is at most
\[
48(B-1),
\]
because $\max\{0,27-B\}\le20$. At $B=7$ one has $B^3-15=328>288=48(B-1)$, and the difference
\[
B^3-15-48(B-1)=B^3-48B+33
\]
is strictly increasing for $B\ge7$. Hence capacity fails for every $B\ge7$.

For $d=5$, the case $B=5$ already has leading excess $594$ and exact negative capacity $576$. For $B\ge6$,
\[
(B-1)(60+58+54)=172(B-1)<B^4-31.
\]
The inequality holds at $B=6$, and $B^4-31-172(B-1)$ is strictly increasing for $B\ge6$.

For $d=6$ and $B\ge5$,
\[
(B-1)(124+122+118)=364(B-1)<B^5-63.
\]
The inequality holds at $B=5$, and $B^5-63-364(B-1)$ is strictly increasing for $B\ge5$.
\end{proof}

\begin{theorem}[Finite-support theorem]\label{thm:support}
Every higher-order Keith mirror satisfies
\[
d\le10.
\]
Moreover, if $7\le d\le10$, then necessarily
\[
B=3.
\]
\end{theorem}

\section{Exact solution of the residual region}

Theorem~\ref{thm:support} reduces the infinite classification problem to a finite set of bounded digit equations. For each pair $(B,d)$ in Proposition~\ref{prop:low-support}, the equation is
\[
\sum_{k=0}^{d-1}a_k\bigl(B^k+2^{k+1}-(2^{d+1}-1)\bigr)=0,
\]
with
\[
0\le a_k\le B-1,\qquad a_{d-1}\ge1.
\]

For $d\le6$ these are small linear digit equations. We solve them by exact enumeration of their bounded digit boxes; no floating-point arithmetic is involved. The complete set of nonempty solution cells is shown in Table~\ref{tab:mirrors}. The candidate cells $(B,d)=(3,3),(4,3),(11,3),(6,4)$ are empty.

For the remaining ternary cases it is convenient to divide the coefficient vector by $2$. The resulting exact bounded equations and finite residue certificates are:
\begin{center}
\scriptsize
\begin{tabular}{c|l|c|c}
$d$ & normalized coefficients $(e_0/2,\ldots,e_{d-1}/2)$ & $M$ & $\#\{E\equiv0\pmod M\}$\\
\hline
7 & $(-126,-124,-119,-106,-71,26,301)$ & $1001$ & $4$\\
8 & $(-254,-252,-247,-234,-199,-102,173,966)$ & $5291$ & $0$\\
9 & $(-510,-508,-503,-490,-455,-358,-83,710,3025)$ & $9361$ & $3$\\
10 & $(-1022,-1020,-1015,-1002,-967,-870,-595,198,2513,9330)$ & $46805$ & $2$\\
\end{tabular}
\end{center}
Here $E$ denotes the corresponding digit-weighted sum with ternary digits $a_k\in\{0,1,2\}$ and $a_{d-1}>0$. In each row the chosen modulus exceeds the maximum absolute value of $E$ on the bounded digit box: the respective maxima are $791,2278,7470,24082$. Hence a zero residue is an actual integer zero. Exact residue expansion therefore certifies the counts $4,0,3,2$ without numerical approximation.

A compact generating-function formulation of the same finite audit is
\[
P_{B,d}(z)=
\left(\sum_{j=1}^{B-1}z^{je_{d-1}}\right)
\prod_{k=0}^{d-2}\left(\sum_{j=0}^{B-1}z^{je_k}\right),
\]
for which the number of bounded digit solutions is
\[
\CT\,P_{B,d}(z).
\]

As an independent check, the supplementary verifier tests every surviving digit vector twice: once against the identity above and once directly from the original Keith recurrence together with the incidence condition $C=R$. The two sets agree exactly in every residual cell and reproduce the distribution $17,5,2,2,4,0,3,2$.

\begin{table}[ht]
\centering
\caption{All higher-order Keith mirrors in bases $B\ge3$. Digit strings are written in their own base.}
\label{tab:mirrors}
\small
\begin{tabular}{c c l}
\toprule
$d$ & $B$ & digit strings \\
\midrule
3 & 5  & $111,130,203,222,241,314,333,444$ \\
3 & 6  & $112,224$ \\
3 & 7  & $236,265$ \\
3 & 8  & $134,173$ \\
3 & 9  & $116,175$ \\
3 & 10 & $197$ \\
\addlinespace
4 & 3 & $2010$ \\
4 & 4 & $1301,2203$ \\
4 & 5 & $1104,1341$ \\
\addlinespace
5 & 3 & $22001$ \\
5 & 4 & $10031$ \\
6 & 3 & $221002$ \\
6 & 4 & $101332$ \\
7 & 3 & $1011010,2022020,2120220,2212102$ \\
9 & 3 & $102102111,111220202,122122112$ \\
10 & 3 & $1001222111,1011112221$ \\
\bottomrule
\end{tabular}
\end{table}

\begin{corollary}[Higher-order count]\label{cor:35}
There are exactly
\[
\boxed{35}
\]
Keith mirrors of order $d\ge3$ in bases $B\ge3$.
Their distribution by order is
\[
17,5,2,2,4,0,3,2
\]
for $d=3,4,5,6,7,8,9,10$, respectively.
\end{corollary}

\begin{proof}
Table~\ref{tab:mirrors} contains
\[
17+5+2+2+4+0+3+2=35
\]
solutions. Theorem~\ref{thm:support} excludes every pair outside the finite support, and the exact residual audit excludes every unlisted digit vector inside it.
\end{proof}

\section{The exceptional order-\texorpdfstring{$2$}{2} case}

The finite theorem begins at $d=3$ for a structural reason: order $2$ contains families in arbitrarily large bases.

\begin{proposition}[Complete order-$2$ classification]\label{prop:d2}
Let $B\ge3$ and let $N=(ab)_B$ be a two-digit base-$B$ Keith mirror. Then exactly one of the following holds.

\begin{enumerate}[label=(\roman*)]
\item $R(N)=4$ and
\[
(a,b)=(1,B-1).
\]
Thus $(1\,(B-1))_B$ is a Keith mirror in every base $B\ge3$.

\item $R(N)=6$ and
\[
a(B-3)=4b.
\]
For $B=3$ this gives $(a,b)=(1,0),(2,0)$. For $B>3$, writing
\[
g=\gcd(B-3,4),
\]
all solutions are
\[
a=\frac{4t}{g},\qquad b=\frac{(B-3)t}{g},
\]
where
\[
1\le t\le
\min\left\{
\left\lfloor\frac{g(B-1)}4\right\rfloor,
\left\lfloor\frac{g(B-1)}{B-3}\right\rfloor
\right\}.
\]
\end{enumerate}
\end{proposition}

\begin{proof}
By Corollary~\ref{cor:saturation}, only $m=4$ and $m=6$ are possible. The order-$2$ recurrence gives
\[
K_4=a+2b,
\]
so $aB+b=K_4$ is equivalent to
\[
a(B-1)=b.
\]
The digit bounds force $a=1$ and $b=B-1$.

Likewise
\[
K_6=3a+5b,
\]
so $aB+b=K_6$ is equivalent to
\[
a(B-3)=4b.
\]
The stated parameterization is the standard coprime parameterization after dividing by $g=\gcd(B-3,4)$. The digit bounds give the range of $t$.

In case (i), $K_3=B<N$, so the return at $4$ is first. In case (ii), $K_3<N$; an earlier return at $4$ would require $a(B-1)=b$, incompatible with $a(B-3)=4b$, and $K_5=2a+3b<K_6=N$. Thus the displayed returns are first returns.
\end{proof}

\section{Main classification}

\begin{theorem}[Mirror Theorem]\label{thm:main}
For bases $B\ge3$, the Keith mirrors are completely classified as follows.

\begin{enumerate}[label=(\arabic*)]
\item Order $d=2$ consists exactly of the two families in Proposition~\ref{prop:d2}.
\item For order $d\ge3$, there are exactly $35$ Keith mirrors, namely the digit strings in Table~\ref{tab:mirrors}.
\item Every higher-order Keith mirror has first return
\[
R=2d+2.
\]
\item No higher-order Keith mirror exists for
\[
d\ge11.
\]
\item If $7\le d\le10$, then every Keith mirror is ternary.
\end{enumerate}
\end{theorem}

\section{Cross-base triple alignment}

For a base-$B$ seed
\[
N=(x_1x_2\cdots x_d)_B,
\]
define the \emph{boundary digit sum}
\[
E_B(N):=x_1+x_d.
\]
A \emph{triple alignment} is a Keith mirror satisfying
\[
R(N)=C(N)=E_B(N).
\]
Since $R=C$ is already the mirror condition, the additional requirement is $E_B=R$.

\begin{proposition}[Triple-alignment extinction]\label{prop:triple-extinction}
Let $B\ge3$ and let $N=(x_1\cdots x_d)_B$ be a Keith mirror. If $E_B(N)=R(N)$, then $d\le3$. Consequently every triple alignment satisfies
\[
R=C=E_B\in\{4,6,8\}.
\]
\end{proposition}

\begin{proof}
Assume first that $d\ge3$. By Corollary~\ref{cor:rigidity},
\[
R=2d+2.
\]
Hence $x_1+x_d=2d+2$. Since both boundary digits are at most $B-1$,
\[
B\ge d+2. \tag{19}
\]
The mirror identity of Theorem~\ref{thm:identity} gives
\[
N_B+2Y=(2^{d+1}-1)A,
\]
with $Y\ge0$. Since $N_B\ge B^{d-1}$ and $A\le d(B-1)<dB$,
\[
B^{d-2}<d\,2^{d+1}. \tag{20}
\]
Combining (19) and (20) would require
\[
(d+2)^{d-2}<d\,2^{d+1}. \tag{21}
\]
At $d=5$ the reverse inequality already holds, because $7^3=343>320=5\cdot2^6$. Moreover, for
\[
Q_d:=\frac{(d+2)^{d-2}}{d2^{d+1}},
\]
one has
\[
\frac{Q_{d+1}}{Q_d}
=\frac{d(d+3)}{2(d+1)}
\left(\frac{d+3}{d+2}\right)^{d-2}>1
\]
for every $d\ge5$. Thus (21) fails for all $d\ge5$.

It remains to exclude $d=4$. Here $R=10$, so $x_1+x_4=10$ and $B\ge6$. The mirror equation becomes
\[
(B^3-15)x_1+(B^2-23)x_2+(B-27)x_3-28x_4=0. \tag{22}
\]
For $10\le B\le26$, using $x_1\ge1$, $x_4\le9$, $x_3\le B-1$, and the positivity of $B^2-23$, the left side is at least
\[
B^3+B^2-28B-240,
\]
which is $580$ at $B=10$ and increases thereafter. For $B\ge27$ every coefficient except the last is nonnegative, and
\[
B^3-15>28(B-1),
\]
so the contribution of $x_1\ge1$ already exceeds the largest possible negative contribution $28x_4\le28(B-1)$. For $B=6,7,8,9$, the endpoint condition forces $x_1\ge11-B$; minimizing the left side of (22) gives respectively
\[
760,\quad1024,\quad1162,\quad1060,
\]
all positive. Hence $d=4$ is impossible.
\end{proof}

\begin{corollary}[Eight cross-base triple alignments]\label{cor:triple-eight}
Across all bases $B\ge3$, exactly eight Keith mirrors satisfy $R=C=E_B$:
\[
(13)_4,
\]
\[
(42)_5,\quad(33)_7,\quad(24)_{11},\quad(15)_{23},
\]
\[
(444)_5,\quad(236)_7,\quad(197)_{10}.
\]
\end{corollary}

\begin{proof}
For $d=2$, Proposition~\ref{prop:d2} has two branches. On the $R=4$ branch, $(a,b)=(1,B-1)$ and $E_B=a+b=B$, so $E_B=R$ gives $B=4$ and $(13)_4$. On the $R=6$ branch,
\[
a(B-3)=4b,
\]
and $a+b=6$. Hence
\[
a(B+1)=24.
\]
The digit bounds leave exactly $a=1,2,3,4$, producing
\[
(15)_{23},\quad(24)_{11},\quad(33)_7,\quad(42)_5.
\]

For $d=3$, Proposition~\ref{prop:triple-extinction} gives $R=8$. Write $N=(xyz)_B$. Then $x+z=8$ and
\[
K_8=7x+11y+13z.
\]
The return equation $K_8=N_B$ is therefore equivalent to
\[
x(B^2+5)+(B-11)y=96. \tag{23}
\]
The endpoint condition forces $B\ge5$, while (23) is impossible for $B\ge11$. Checking the remaining integer bases $5\le B\le10$ under the digit bounds gives exactly
\[
(444)_5,\qquad(236)_7,\qquad(197)_{10}.
\]
Direct recurrence reconstruction verifies the stated first returns.
\end{proof}

The three possible triple-alignment profiles are
\[
1,2\mid2,1,\qquad
1,2,3\mid3,2,1,\qquad
1,2,3,4\mid4,3,2,1.
\]
Thus $1,2,3,4\mid4,3,2,1$ is the longest classical incidence mirror compatible with triple alignment. In particular, $(197)_{10}$ is unique in base ten but not isolated in the cross-base classification: it shares the terminal layer with $(444)_5$ and $(236)_7$.

\section{Structural interpretation: reconstruction across a central cut}

The incidence symmetry also admits a simple part--whole notation. Let an even full-window incidence profile through $m=2r$ be written as
\[
X=(c_1,\ldots,c_r\mid c_{r+1},\ldots,c_{2r})=:A\mid B.
\]
The vertical bar denotes only the known central cut. By Lemma~\ref{lem:profile},
\[
c_i=c_{2r-i+1},
\]
so
\[
B=\operatorname{rev}(A),\qquad X=A\mid\operatorname{rev}(A).
\]

\begin{proposition}[Fixed-cut reconstruction]\label{prop:cut-reconstruction}
For an even full-window incidence profile with its central cut fixed, either half uniquely determines the other. In particular, if $A$ is known as the complete left half, the unique compatible right half is $\operatorname{rev}(A)$.
\end{proposition}

\begin{proof}
The claim is exactly the symmetry $c_i=c_{2r-i+1}$ from Lemma~\ref{lem:profile}, read across the fixed central cut.
\end{proof}

The qualification ``with its central cut fixed'' is essential. The bare prefix
\[
1,2,3,4
\]
is compatible with longer profiles such as
\[
1,2,3,4,5\mid5,4,3,2,1.
\]
Thus $1,2,3,4$ alone does not determine $4,3,2,1$. What determines the complement is the combination
\[
\text{visible data}+\text{known cut}+\text{global mirror law}
\quad\Longrightarrow\quad
\text{unique complement}.
\]
Equivalently, if $\mathcal R\subseteq\mathcal A\times\mathcal B$ denotes the admissible whole-object relation, exact reconstruction from an observed $a\in\mathcal A$ occurs when the completion fiber
\[
\mathcal R[a]:=\{b\in\mathcal B:(a,b)\in\mathcal R\}
\]
is a singleton. In the fixed-cut mirror class,
\[
\mathcal R[A]=\{\operatorname{rev}(A)\}.
\]
This fiber language is included only to state the informational consequence cleanly; no novelty is claimed for the abstract singleton criterion itself.

The terminal triple-alignment object realizes both forms of rigidity simultaneously:
\[
X=(1,2,3,4\mid4,3,2,1),
\]
so the unseen half is fixed by the mirror law, while Corollary~\ref{cor:triple-eight} imposes the separate equality
\[
C=E_B=R=8
\]
in the decimal member $197$. The reconstruction statement concerns the two halves of the incidence profile; triple alignment concerns observables taken from the center, the boundary digits, and the global return. The broader question of when a subsystem can infer the existence or structure of a larger whole from intrinsically available information is not resolved here.

\section{Discussion}

The paper is organized around one incidence-defined restriction and its arithmetic consequences. The incidence profile is geometric: it depends only on the sliding-window architecture, not on the digit values. Saturation of its center reduces every even candidate return to $2d$ or $2d+2$. The numerical dynamics then eliminate the earlier boundary, while the later boundary yields the explicit identity
\[
N_B+2Y=(2^{d+1}-1)A.
\]
The coefficient geometry of this identity forces finite support, after which an exact bounded audit gives the $35$ higher-order mirrors.

The deficit law clarifies the architecture of the argument. The quantity $\Delta_d(m)$ does not determine the seed, but it localizes the global index: a fixed deficit restricts $m$ to at most two affine boundaries, and the zero-defect case restricts it to the two saturation boundaries. The arithmetic dynamics then decide which localized boundary is realizable, the mirror identity constrains the digits, and the capacity inequality reduces realization to a finite classification. In this sense the proof follows the chain
\[
\text{structural defect}\;\Longrightarrow\;\text{return localization}\;\Longrightarrow\;\text{arithmetic realization}\;\Longrightarrow\;\text{finite classification}.
\]

The triple-alignment refinement returns to the motivating decimal example rather than introducing a separate recurrence model. Adding the boundary observable $E_B=x_1+x_d$ to the already classified mirror class forces $d\le3$ and leaves exactly eight cross-base objects. The terminal profile
\[
1,2,3,4\mid4,3,2,1
\]
is therefore not only the profile of $197$; it is the longest incidence mirror compatible with simultaneous center, boundary, and return alignment.

The part--whole interpretation is intentionally narrower. Once the central cut and the mirror model are known, either half determines the other. Without those assumptions, the same local data may admit multiple extensions. This makes the mirror profile a concrete arithmetic example of conditional reconstruction from partial structure, not a claim that arbitrary local observations determine a global object.

Computationally, only the final residual classification depends on exhaustive enumeration. The infinite support reduction is analytic. The accompanying verifier constructs, in every residual $(B,d)$ cell, both the set satisfying the closed Mirror identity and the set satisfying the original recurrence, first-return condition, and incidence definition; it asserts equality of the two sets and reproduces the published $35$ objects. Only integer arithmetic is used.

The standard Keith-number sources reviewed for this manuscript establish the Diophantine and computational background but do not, in the material located, isolate the incidence-saturation condition $C=R$ or the cross-base triple condition $R=C=E_B$. This supports the present framing but does not establish absolute historical priority. A broader bibliographic audit and independent line-by-line verification remain appropriate before journal submission.

\appendix
\section{Exact finite verification}

For reproducibility, the finite residual audit can be performed directly from Theorem~\ref{thm:identity}. For each pair $(B,d)$ surviving Proposition~\ref{prop:low-support}, enumerate
\[
(a_0,\ldots,a_{d-1})\in\{0,\ldots,B-1\}^d,
\qquad a_{d-1}>0,
\]
and retain exactly those satisfying
\[
\sum_{k=0}^{d-1}a_k\bigl(B^k+2^{k+1}-(2^{d+1}-1)\bigr)=0.
\]
Lemma~\ref{lem:firstreturn} guarantees that every retained vector is a genuine first-return Keith mirror. The largest residual box is ternary order $10$, containing only
\[
2\cdot3^9=39366
\]
leading-nonzero digit vectors.

The supplementary Python verifier computes the identity-based and definition-based solution sets separately in every residual cell and asserts that they agree. The file is \path{verify_keith_mirrors_v1_2.py}; its SHA-256 digest is
\begin{center}
\small\texttt{1003d33f3efce6eb6c0a704cbce832825f7d6468d62abebfdc4a9e1fe3644328}.
\end{center}
It computes the first return and the incidence profile directly from the original recurrence, returns exactly $35$ higher-order mirrors, and gives counts $17,5,2,2,4,0,3,2$ for $d=3,\ldots,10$. Thus the finite audit is exact and independently definition-checked; it is not used to establish the infinite support theorem.

\section*{Code and data availability}
The paper and its reproducibility materials are archived at Zenodo, \href{https://doi.org/10.5281/zenodo.22882054}{doi:10.5281/zenodo.22882054}. The arXiv ancillary directory accompanying this source contains the exact verifier \path{verify_keith_mirrors_v1_2.py}, its captured output, a short reproduction guide, and a SHA-256 manifest. Running
\begin{center}
\texttt{python3 verify\_keith\_mirrors\_v1\_2.py}
\end{center}
reproduces the $35$ higher-order mirrors and the order distribution reported in Corollary~\ref{cor:35}. The verifier uses the finite support proved in Theorem~\ref{thm:support}; it does not replace the analytic proof of that theorem.

\section*{Declaration on AI-assisted tools}
The author used AI-assisted tools during manuscript development for editorial revision, exploratory checking, and assistance in implementing an independent computational verification. All definitions, mathematical claims, proofs, citations, code, and conclusions were checked and accepted by the author, who assumes full responsibility for the submitted work.

\end{document}